\documentclass[10pt, reqno]{amsart}
\usepackage{mathrsfs}
\usepackage{epsfig}
\usepackage[dvipsnames]{xcolor}
\usepackage{graphicx}
\usepackage{amsmath}
\usepackage{amsfonts}
\usepackage{amssymb}
\usepackage{bbm}
\usepackage{amsthm}
\usepackage{amscd}
\usepackage{verbatim}
\usepackage{tikz}
\usepackage[sc]{mathpazo}
\usepackage[T1]{fontenc} 
\usepackage[autostyle]{csquotes}
\usepackage{hyperref}
\usepackage{ytableau}
\usetikzlibrary{arrows}

\tikzset{
vertex/.style={circle,draw,minimum size=1.5em},
edge/.style={->,> = latex'}
}

\usepackage{mathtools}

\mathtoolsset{showonlyrefs=true}

\theoremstyle{plain}
\newtheorem{theorem}{Theorem}[section]
\newtheorem{proposition}[theorem]{Proposition}

\newtheorem{corollary}[theorem]{Corollary}

\theoremstyle{definition}

\newtheorem{remark}[theorem]{Remark}
\newtheorem{example}[theorem]{Example}

\numberwithin{equation}{section}

\DeclareMathAlphabet{\mathpzc}{OT1}{pzc}{m}{it}

\definecolor{MyBlue1}{RGB}{229, 119, 167}
\definecolor{MyBlue2}{RGB}{61, 101, 165}
\definecolor{MyBlue3}{RGB}{124, 161, 204}

\newcommand{\tr}{\operatorname{tr}}

\newcommand{\norm}[1]{\| #1 \|}

\let\oldenumerate=\enumerate
\def\enumerate{
\oldenumerate
\setlength{\itemsep}{5pt}
}
\let\olditemize=\itemize
\def\itemize{
\olditemize
\setlength{\itemsep}{5pt}
}

\usepackage{hyperref}

\begin{document}

\title[Extremal polynomial norms of graphs]{Extremal polynomial norms of graphs}
\author[\'A. Ch\'avez]{\'Angel Ch\'avez}
\author[S. Fullerton]{Sarah Fullerton}
\author[S. LaFortune]{Sam LaFortune}
\author[K. Linarez]{Keyron Linarez}
\author[N. Liyanage]{Nethmin Liyanage}
\author[J. Son]{Justin Son}
\author[T. Ting]{Tyler Ting}


\begin{abstract}
Recent work shows that a new family of norms on Hermitian matrices arise by evaluating the even degree complete homogeneous symmetric (CHS) polynomials on the eigenvalues of a Hermitian matrix. The \emph{CHS norm} of a graph is then defined by evaluating the even degree CHS polynomials on the eigenvalues of the adjacency matrix of a graph. The fact that these norms are defined in terms of eigenvalues (as opposed to singular values) ensures they can distinguish between graphs that other norms cannot. In addition, we prove that the CHS norms are minimized over all connected graphs by the path and maximized over all connected graphs by the complete graph. Finally, we prove that the CHS norms are minimized over all trees by the path and maximized over all trees by the star. Our paper is intended for a wide mathematical audience and we assume no prior knowledge about graphs or symmetric polynomials.
\end{abstract}

\maketitle

\section{Introduction}






Suppose $G$ is a graph of order $n$. The \emph{singular values} of $G$ are the singular values of the adjacency matrix $A(G)$ and we write them in nonincreasing order $\sigma_1\geq \sigma_2\geq \cdots \geq \sigma_n$. The adjacency matrix $A(G)$ is normal since it is real and symmetric. Therefore, the singular values of $G$ correspond to the absolute values of the eigenvalues of $A(G)$ \cite[Theorem 15.3.4]{GarciaHorn}.

\begin{example}\label{ex:IntroComplete}
The adjacency matrix of the complete graph $G$ of order $3$ is 
\begin{align*}
A(G)=\begin{bmatrix}
0 & 1 & 1 \\
1 & 0 & 1 \\
1 & 1 & 0 
\end{bmatrix}.
\end{align*} The eigenvalues of $A(G)$ are $2, -1, -1$. Consequently, the singular values of $G$ are given by $\sigma_1=2$ and $\sigma_2=\sigma_3=1$.
\end{example}

The adjacency matrix of a graph $G$ is real and symmetric. Therefore, a norm on the space of real symmetric matrices gives us a natural way to measure the ``size'' of $G$. There are several examples of such norms:\\

\noindent\underline{\sc Graph energy}. The \emph{energy} of $G$ is defined as the sum $\norm{G}_{\ast}=\sigma_1+\sigma_2+\cdots +\sigma_n$
of the singular values of $G$ and was introduced by Gutman in 1978 \cite{Gutman1}.\\

\noindent\underline{\sc Spectral norm}. The \emph{spectral norm} of $G$ is defined as $\norm{G}=\sigma_1.$ The Perron-Frobenius theorem ensures that $\norm{G}$ equals the largest eigenvalue of $A(G)$.\\

\noindent\underline{\sc Ky Fan norms}. The \emph{Ky Fan $k$-norm} of $G$ is the sum $\norm{G}_{\tiny\mbox{KF}(k)}=\sigma_1+\sigma_2+\cdots +\sigma_k$ of the $k$ largest singular values of $G$. We remark $\norm{G}_{\tiny\mbox{KF}(n)}=\norm{G}_{\ast}$. \\

\noindent\underline{\sc Schatten norms}. The \emph{Schatten $p$-norm} of $G$ is $\norm{G}_{S_p}=\big( \sigma_1^p+\sigma_2^p+\cdots +\sigma_n^p \big)^{\frac{1}{p}}$ for all real numbers $p\geq 1$. We remark $\norm{G}_{S_1}=\norm{G}_{\ast}$.\\

\begin{example}
Consider the graph $G$ of Example \ref{ex:IntroComplete}. Various norms of $G$ are
\begin{center}
$\norm {G}_{\ast} = 4,\quad \norm {G} = 2,\quad \norm {G}_{S_2} = \sqrt{6}$
\end{center} 
\end{example}

There are hundreds of research articles devoted to the study of $\norm{G}_{\ast}$ and $\norm{G}$. And three decades of progress on the energy of graphs, for instance, is summarized in \cite{GutmanLiShi}. The lasting interest in $\norm{G}_{\ast}$ and $\norm{G}$ has more recently expanded to include the study of more obscure norms of graphs. Nikiforov establishes several key results for the Ky Fan and Schatten norms of graphs, for instance, in \cite{Nikiforov1, Nikiforov2}.

\subsection{The CHS norms of graphs}
A new family of norms on the space of real symmetric matrices was introduced in \cite{Aguilar}. This family of norms are defined by evaluating the complete homogeneous symmetric polynomials of even degree on the eigenvalues of a real symmetric matrix. The \emph{complete homogeneous symmetric} (CHS) polynomial of degree $d$ in the variables $x_1, x_2, \ldots, x_n$ is defined as the sum
\begin{align*}
h_d(x_1, x_2, \ldots, x_n)=\sum_{1\leq i_1\leq i_2\leq\cdots \leq i_d\leq n}x_{i_1}x_{i_1}\cdots x_{i_d}
\end{align*}of all degree-$d$ monomials in $x_1, x_2, \ldots, x_n$ \cite[Section 7.5]{Stanley2}. A useful description of the CHS polynomials is given by their generating function. In particular, the CHS polynomials satisfy the following relation \cite[Equation 7.11]{Stanley2}:

\begin{align}
\sum_{d=0}^{\infty}h_d(x_1, x_2, \ldots, x_n)t^d=\prod_{i=1}^n \frac{1}{1-x_it}.\label{eq:MGF}
\end{align}

\begin{example}\label{ex:CHS}
A few examples of CHS polynomials in the variables $x,y,z$ are 
\begin{align*}
h_1(x,y,z)&= x + y + z, \\
h_2(x,y,z)&= x^2 + y^2 +z^2+xy + yz + zx, \\
h_3(x,y,z)&= x^3 + y^3+z^3+x^2y + x^2z + y^2x + y^2z + z^2x + z^2y + xyz.
\end{align*} 
\end{example}

Let $\mathrm{H}_n$ denote the real space of $n\times n$ Hermitian matrices. The eigenvalues of each $A\in \mathrm{H}_n$ are real \cite[Theorem 12.6.1]{GarciaHorn} and we write them in nonincreasing order $\lambda_1\geq \lambda_2\geq \cdots \geq \lambda_n$. The following is \cite[Theorem 1]{Aguilar}.

\begin{theorem}\label{thm:Hermitian}
Suppose $d\geq 2$ is even. The following function defines a norm on $\mathrm{H}_n$:
\begin{align*}
\norm{A}_d=\Big( h_d(\lambda_1, \lambda_2, \ldots, \lambda_n)\Big)^{\frac{1}{d}}.
\end{align*} 
\end{theorem}

\begin{remark}
Hunter established positive definiteness of the even degree CHS polynomials on $\mathbb{R}^n$ \cite{Hunter}. Tao further proved that the even degree CHS polynomials are Schur convex (respect the majorization partial order) on $\mathbb{R}^n$ \cite{Tao}. An alternate proof for Schur convexity appears in the proof of \cite[Theorem 1.1 (d)]{Chavez1}. Various proofs of positive definiteness appear in the literature \cite{Barvinok, Baston, Bottcher, GarciaOmar, Roventa, Tao}.
\end{remark}

\begin{remark}
The norms $\norm{\cdot}_d$ of Theorem \ref{thm:Hermitian} extend naturally to norms on the space $\mathrm{M}_n$ of $n\times n$ complex matrices \cite[Theorem 3]{Aguilar}. However, we do not concern ourselves with these more general norms because our graphs are not directed and therefore have Hermitian adjacency matrices.
\end{remark}

\begin{remark}
Theorem \ref{thm:Hermitian} can be interpreted in terms of moments of weighted sums of independent standard exponential random variables \cite[Subsection 2.2]{Aguilar}. A general probabilistic version of Theorem \ref{thm:Hermitian} appears in \cite[Theorem 1.1]{Chavez1}. The proof of \cite[Theorem 1.1]{Chavez1} relies on Lewis' framework for convex matrix analysis \cite{Lewis}. A completely new and much simpler proof of \cite[Theorem 1.1]{Chavez1} appears in the proof of \cite[Theorem 2]{Chavez2}. 
\end{remark}

\noindent\underline{\sc CHS norms}. The space of real symmetric matrices is a subspace of $\mathrm{H}_n$. Consequently, we can use the norms of Theorem \ref{thm:Hermitian} to measure the size of a graph. In particular, suppose $G$ is a graph with adjacency matrix $A(G)$. The \emph{CHS $d$-norm} of $G$ is denoted $\norm{G}_d$ and is defined by setting $\norm{G}_d=\norm{A(G)}_d.$\\

\begin{example}
Consider the graph $G$ of Example \ref{ex:IntroComplete} and the polynomials of Example \ref{ex:CHS}. The eigenvalues of $A(G)$ are $\lambda_1=2$ and $\lambda_2=\lambda_3=-1$. Therefore, 
\begin{align*}
\norm{G}_4=\Big( h_4(\lambda_1, \lambda_2, \lambda_3)\Big)^{\frac{1}{4}}= 9^{\frac{1}{4}} \approx 1.73,\\
\norm{G}_6=\Big( h_6(\lambda_1, \lambda_2, \lambda_3)\Big)^{\frac{1}{6}}= 31^{\frac{1}{6}}\approx 1.77.\\
\end{align*}
\end{example}

An important observation is in order. The energy and spectral, Ky Fan and Schatten norms are defined in terms of singular values. The CHS norms are defined in terms of eigenvalues. The CHS norms can therefore distinguish between graphs that the other norms cannot. The following appears in \cite[Example 2]{Aguilar}.

\begin{example}\label{ex:SingularlyCospectral}
Let $K$ denote the matrix of Example \ref{ex:IntroComplete}. Define $G$ and $H$ to be the graphs whose adjacency matrices are given by
\begin{align*}
A(G)=\begin{bmatrix}
K & 0\\
0 &K
\end{bmatrix},\quad A(H)=\begin{bmatrix}
0 & K\\
K & 0
\end{bmatrix}.
\end{align*}The graph $G$ has eigenvalues $2, 2, -1, -1,-1,-1$ and the graph $H$ has eigenvalues $2, 1, 1, -1, -1, -2$. Consequently, the graphs $G$ and $H$ have the same singular values but distinct eigenvalues. The graph energies, together with the spectral, Ky Fan and Schatten norms of $G$ and $H$ are equal. However, the CHS 6-norms of $G$ and $H$ are distinct. In particular, $\norm{G}_6^6=120$ and $\norm{H}_6^6=112$.
\end{example}

Graphs with the same singular values but different eigenvalues are called \emph{noncospectral singularly cospectral}. The graphs in Example \ref{ex:SingularlyCospectral} are examples of such graphs. Infinitely many pairs of noncospectral singularly cospectral graphs are given in \cite[Theorem 4.1]{Conde}: $G=F\times K_2$ and $H=F\sqcup F$ are noncospectral singularly cospectral whenever $F$ is a nonbipartite graph of order $n\geq 3$. The graphs of Example \ref{ex:SingularlyCospectral} are exactly of this type. In particular, let $F=K_3$.


\subsection{Statement of results} The following theorem ensures that the CHS norms can always distinguish between noncospectral singularly cospectral graphs.

\begin{theorem}\label{thm:MainTheorem1}
If $G$ and $H$ are noncospectral singularly cospectral, then $\norm{G}_d\neq \norm{H}_d$ for some even integer $d\geq 6$.
\end{theorem}

Our main focus is on finding extremal graphs for CHS norms. To begin, suppose $G$ is a simple connected graph of order $n$. What are the smallest and largest values that a given norm of $G$ can take? For instance, the spectral norm of $G$ satisfies the bounds
\begin{align}
2\cos\Big( \frac{\pi}{n+1}\Big)=\norm{P_n}\leq\norm{G}\leq \norm{K_n}=n-1,\label{eq:Graph}
\end{align}in which $P_n$ denotes the path of order $n$ and $K_n$ denotes the complete graph of order $n$. Moreover, the spectral norm of a tree $T$ of order $n$ satisfies the bounds 
\begin{align}
2\cos\Big( \frac{\pi}{n+1}\Big)=\norm{P_n}\leq\norm{T}\leq \norm{S_n}=\sqrt{n-1},\label{eq:Tree}
\end{align} in which $S_n$ denotes the star with $n-1$ leaves. Bounds \eqref{eq:Graph} and \eqref{eq:Tree} were established in 1957 by Collatz and Sinogowitz \cite{Collatz}. The CHS norms satisfy the same bounds.

\begin{theorem}\label{thm:MainTheorem2}
Suppose $d\geq 2$ is even. The CHS norms satisfy the following bounds: \\

\noindent (a) If $G$ is a simple connected graph of order $n$, then $\norm{P_n}_d\leq \norm{G}_d\leq \norm{K_n}_d.$\\

\noindent (b) If $T$ is a tree of order $n$, then $\norm{P_n}_d\leq \norm{T}_d\leq \norm{S_n}_d.$
\end{theorem}

This paper, which is intended for a wide audience, takes us on a short journey through the basics of spectral graph theory and is organized as follows. Section \ref{sec:Basic} introduces the reader to basic terminology in graph theory. Section \ref{sec:Spectra} establishes that the spectrum of any graph is real and provides several important examples. The remarkable link between graph spectra and closed walks is explored in Section \ref{sec:ClosedWalks}. The power sum symmetric polynomials and their relation with the CHS polynomials is discussed in Section \ref{sec:PowerSums}. Theorem \ref{thm:MainTheorem1} is proved in Section \ref{sec:Proof1} and Theorem \ref{thm:MainTheorem2} is proved in Section \ref{sec:Proof2}. Finally, we conclude with possible directions for future work and closing remarks.


\section{Basic graph theory}\label{sec:Basic}

Essentially, a graph is just a collection of vertices connected by paths called edges. More formally, a \emph{graph} of order $n$ is an ordered pair $(V,E)$ consisting of a set $E=E(G)$ of \emph{edges} and a set $V=V(G)$ of \emph{vertices} such that $|V|=n$. We adopt standard notation and let $m=|E(G)|$ denote the number of edges in a graph $G$. A graph is \emph{simple} if it contains no loops or multiple edges and it is \emph{connected} if it contains no isolated vertices. We refer the reader to \cite{Bollobas} for further background.

\begin{example}\label{ex:TwoGraphs}
Figure \ref{fig:Figure2} shows two graphs. Both of them have $n=4$ vertices and both are connected since none of their vertices are isolated. However, only one of the graphs is a simple graph.
\end{example}

\begin{figure}[h]

\begin{center}
\begin{tikzpicture}[main_node/.style={circle,fill=blue!20,draw,minimum size=0.4em,inner sep=1pt}, every loop/.style={}]

\node[main_node] (1) at (0, 0) {};
\node[main_node] (2) at (1,0) {};
\node[main_node] (3) at (1, 1) {};
\node[main_node] (4) at (0, 1) {};

\path[every node/.style={}]
(1) edge node {} (2)
(1) edge node {} (3) 
(1) edge node {} (4) 
(2) edge node {} (3)
(2) edge node {} (4)
(3) edge node {} (4) ;
\end{tikzpicture} 
\quad\quad\quad\quad\quad
\begin{tikzpicture}[main_node/.style={circle,fill=blue!20,draw,minimum size=0.4em,inner sep=1pt}, every loop/.style={}]

\node[main_node] (1) at (0, 0) {};
\node[main_node] (2) at (1,0) {};
\node[main_node] (3) at (1, 1) {};
\node[main_node] (4) at (0, 1) {};

\path[every node/.style={}]
(1) edge node {} (2)
(1) edge node {} (3) 
(1) edge node {} (4) 
(2) edge node {} (3)
(3) edge[loop] node {} (4) ;
\end{tikzpicture}
\caption{Two connected graphs with $n=4$ vertices. The graph on the right is not simple since one of the vertices has a loop.} 
\label{fig:Figure2} 
\end{center}
\end{figure}
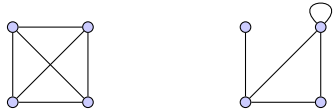

\begin{example}\label{ex:NotConnected}
Figure \ref{fig:Figure3} shows a graph with $n=5$ vertices with $m=4$ edges. The graph is simple since it has no loops or multiple edges. However, it is not connected since there is an isolated vertex.
\end{example}

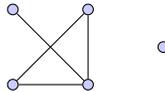
\begin{figure}[h]

\begin{center}
\begin{tikzpicture}[main_node/.style={circle,fill=blue!20,draw,minimum size=0.4em,inner sep=1pt}, every loop/.style={}]

\node[main_node] (1) at (0, 0) {};
\node[main_node] (2) at (1,0) {};
\node[main_node] (3) at (1, 1) {};
\node[main_node] (4) at (0, 1) {};
\node[main_node] (5) at (2, 0.5) {};

\path[every node/.style={}]
(1) edge node {} (2)
(1) edge node {} (3) 
(2) edge node {} (3)
(2) edge node {} (4) ;
\end{tikzpicture} 
\caption{A graph with $n=5$ vertices that is not connected.} 
\label{fig:Figure3}
\end{center}
\end{figure}

The vertices of the graphs in Examples \ref{ex:TwoGraphs} and \ref{ex:NotConnected} are not labelled and no explicit descriptions of $V$ or $E$ are given. However, it is often useful to label vertices and give explicit descriptions of $V$ and $E$. In particular, suppose $G$ is a graph of order $n$ with vertex set $V=\{v_1, v_2, \ldots, v_n\}$. Edge $e_{ij}\in E$ if and only if there is an edge connecting $v_i$ and $v_j$. We remark that $e_{ij}=e_{ji}$ since the edges in our graphs have no direction associated to them.

\begin{example}\label{ex:Path}
The graph in Figure \ref{fig:Figure4} has vertex set $V=\{v_1, v_2, v_3\}$. The edge set is $E=\{e_{12}, e_{23}\}$ because there is an edge joining $v_1$ and $v_2$ and an edge joining $v_2$ and $v_3$.
\end{example}

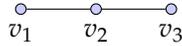
\begin{figure}[h]
\begin{center}
\begin{tikzpicture}[main_node/.style={circle,fill=blue!20,draw,minimum size=0.4em,inner sep=1pt}, every loop/.style={}]

\node[main_node, label=below: $v_1$ ] (1) at (0, 0){};
\node[main_node, label=below: $v_2$] (2) at (1,0) {};
\node[main_node, label=below: $v_3$] (3) at (2, 0) {};

\path[every node/.style={}]
(1) edge (2)
(2) edge (3) ;
\end{tikzpicture}
\caption{A connected graph with $n=3$ vertices and $m=2$ edges.} 
\label{fig:Figure4}
\end{center}
\end{figure}


\subsection{Important families of graphs}\label{sub:Families} Here and throughout we assume all graphs are simple and connected. There are several important families of such graphs. The focus of our paper, however, is on the families of paths, stars, complete and complete bipartite graphs. We define these families below.\\

\noindent\underline{\sc Path graphs}. The \emph{path graph} of order $n$ is denoted $P_n$ and is defined according to the sets $V=\{v_1, v_2, \ldots, v_n\}$ and $E=\{e_{12}, e_{23}, \cdots, e_{n-1 n}\}$. A few examples of paths are shown in Figure \ref{fig:Figure5}. \\

\noindent\underline{\sc Complete graphs}. The \emph{complete graph} $K_n$ is the graph of order $n$ with the maximum possible number of edges. In particular, $K_n$ has $m={n\choose 2}$ edges. Figure \ref{fig:Figure6} shows complete graphs $K_n$ of various size. \\

\noindent\underline{\sc Complete bipartite graphs}. The \emph{complete bipartite graph} $K_{m,n}$ has $m+n$ vertices separated into two sets of size $m$ and $n$. No edges connect vertices belonging to the same set and there exists and edge between any two vertices belonging to different sets. Figure \ref{fig:Figure7} shows the complete bipartite graphs $K_{2,2}$ and $K_{2,3}$. \\

\noindent\underline{\sc Star graphs}. The \emph{star graph} with $n-1$ leaves is denoted $S_n$ and consists of a single central vertex with $n-1$ edges connected to it. In particular, $S_n=K_{n-1,1}$. Figure \ref{fig:Figure8} shows stars $S_n$ of various size.\\

\begin{figure}[h]
\begin{center}
{\footnotesize
\begin{tikzpicture}[main_node/.style={circle,fill=blue!20,draw,minimum size=0.5em,inner sep=1pt}, every loop/.style={}]

\node[main_node] (1) at (0, 0){};
\node[main_node] (2) at (0.75,0) {};

\path[every node/.style={}]
(1) edge (2) ;
\end{tikzpicture}
\quad\quad\quad\quad
\begin{tikzpicture}[main_node/.style={circle,fill=blue!20,draw,minimum size=0.5em,inner sep=1pt}, every loop/.style={}]

\node[main_node] (1) at (0, 0){};
\node[main_node] (2) at (0.75,0) {};
\node[main_node] (3) at (1.5, 0) {};

\path[every node/.style={}]
(1) edge (2)
(2) edge (3) ;
\end{tikzpicture}
\quad\quad\quad\quad
\begin{tikzpicture}[main_node/.style={circle,fill=blue!20,draw,minimum size=0.5em,inner sep=1pt}, every loop/.style={}]

\node[main_node] (1) at (0, 0){};
\node[main_node] (2) at (0.75,0) {};
\node[main_node] (3) at (1.5, 0) {};
\node[main_node] (4) at (2.25, 0) {};

\path[every node/.style={}]
(1) edge (2)
(2) edge (3)
(3) edge (4) ;
\end{tikzpicture}
}
\caption{Various paths $P_n$.} 
\label{fig:Figure5}
\end{center}
\end{figure}
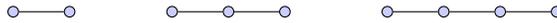

\begin{figure}[h]
\begin{center}
{\footnotesize
\begin{tikzpicture}[main_node/.style={circle,fill=blue!20,draw,minimum size=0.5em,inner sep=1pt}, every loop/.style={}]

\node[main_node] (1) at (-0.65,-0.375) {};
\node[main_node] (2) at (0, 0.75) {};
\node[main_node] (3) at (0.65, -0.375) {};

\path[every node/.style={}]
(1) edge (2)
(1) edge (3)
(2) edge (3) ;
\end{tikzpicture}
\quad\quad\quad\quad
\begin{tikzpicture}[main_node/.style={circle,fill=blue!20,draw,minimum size=0.5em,inner sep=1pt}, every loop/.style={}]

\node[main_node] (1) at (-0.53,-0.53) {};
\node[main_node] (2) at (-0.53, 0.53) {};
\node[main_node] (3) at (0.53, -0.53) {};
\node[main_node] (4) at (0.53, 0.53) {};

\path[every node/.style={}]
(1) edge (2)
(1) edge (3)
(1) edge (4) 
(2) edge (3)
(2) edge (4)
(3) edge (4) ;
\end{tikzpicture}
\quad\quad\quad\quad
\begin{tikzpicture}[main_node/.style={circle,fill=blue!20,draw,minimum size=0.5em,inner sep=1pt}, every loop/.style={}]

\node[main_node] (1) at (0, 0.75) {};
\node[main_node] (2) at (0.71, 0.23) {};
\node[main_node] (3) at (-0.71, 0.23) {};
\node[main_node] (4) at (-0.441, -0.607) {};
\node[main_node] (5) at (0.441, -0.607) {};

\path[every node/.style={}]
(1) edge (2)
(1) edge (3)
(1) edge (4) 
(1) edge (5)
(2) edge (3)
(2) edge (4) 
(2) edge (5)
(3) edge (4)
(3) edge (5) 
(4) edge (5);
\end{tikzpicture}
}
\caption{Various complete graphs $K_n$.} 
\label{fig:Figure6}
\end{center}
\end{figure}
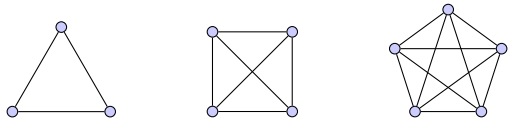

\begin{figure}[h]
\begin{center}
{\footnotesize
\begin{tikzpicture}[main_node/.style={circle,fill=blue!20,draw,minimum size=0.5em,inner sep=1pt}, every loop/.style={}]

\node[main_node] (1) at (0, 0) {};
\node[main_node] (2) at (1,0) {};
\node[main_node] (3) at (0,1) {};
\node[main_node] (4) at (1,1) {};

\path[every node/.style={}]
(1) edge (3)
(1) edge (4)
(2) edge (3) 
(2) edge (4);
\end{tikzpicture}\quad\quad\quad\quad
\begin{tikzpicture}[main_node/.style={circle,fill=blue!20,draw,minimum size=0.5em,inner sep=1pt}, every loop/.style={}]

\node[main_node] (1) at (-1,0) {};
\node[main_node] (2) at (0, 0) {};
\node[main_node] (3) at (1,0) {};
\node[main_node] (4) at (-0.5,1) {};
\node[main_node] (5) at (0.5,1) {};

\path[every node/.style={}]
(4) edge (1)
(4) edge (2)
(4) edge (3) 
(5) edge (1)
(5) edge (2)
(5) edge (3) ;
\end{tikzpicture}
}
\caption{The complete bipartite graphs $K_{2,2}$ (left) and $K_{2,3}$ (right).} 
\label{fig:Figure7}
\end{center}
\end{figure}
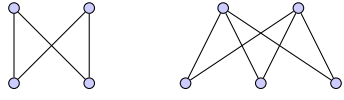

\begin{figure}[h]
\begin{center}
{\footnotesize
\begin{tikzpicture}[main_node/.style={circle,fill=blue!20,draw,minimum size=0.5em,inner sep=1pt}, every loop/.style={}]

\node[main_node] (1) at (0, 0){};
\node[main_node] (2) at (-0.65,-0.375) {};
\node[main_node] (3) at (0, 0.75) {};
\node[main_node] (4) at (0.65, -0.375) {};

\path[every node/.style={}]
(1) edge (2)
(1) edge (3)
(1) edge (4) ;
\end{tikzpicture}
\quad\quad\quad\quad
\begin{tikzpicture}[main_node/.style={circle,fill=blue!20,draw,minimum size=0.5em,inner sep=1pt}, every loop/.style={}]

\node[main_node] (1) at (-0.53,-0.53) {};
\node[main_node] (2) at (-0.53, 0.53) {};
\node[main_node] (3) at (0.53, -0.53) {};
\node[main_node] (4) at (0.53, 0.53) {};
\node[main_node] (5) at (0,0) {};

\path[every node/.style={}]
(5) edge (1)
(5) edge (2)
(5) edge (3)
(5) edge (4);

\end{tikzpicture}
\quad\quad\quad\quad
\begin{tikzpicture}[main_node/.style={circle,fill=blue!20,draw,minimum size=0.5em,inner sep=1pt}, every loop/.style={}]

\node[main_node] (1) at (0, 0.75) {};
\node[main_node] (2) at (0.71, 0.23) {};
\node[main_node] (3) at (-0.71, 0.23) {};
\node[main_node] (4) at (-0.441, -0.607) {};
\node[main_node] (5) at (0.441, -0.607) {};
\node[main_node] (6) at (0,0) {};

\path[every node/.style={}]
(6) edge (1)
(6) edge (2)
(6) edge (3)
(6) edge (4)
(6) edge (5);

\end{tikzpicture}
}
\caption{Various stars $S_n$.} 
\label{fig:Figure8} 
\end{center}
\end{figure}
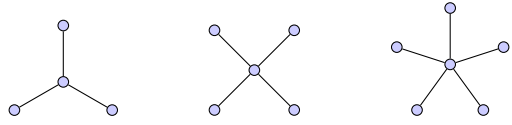


\section{Graph spectra}\label{sec:Spectra}

Suppose $G$ is a graph of order $n$. The \emph{adjacency matrix} of $G$ is the $n\times n$ matrix $A(G)$ whose entries are defined by setting $[A(G)]_{ij}=1$ if vertices $v_i$ and $v_j$ are connected by an edge and $[A(G)]_{ij}= 0$ otherwise. We sometimes write $A=A(G)$ when $G$ is understood. The \emph{spectrum} $\operatorname{spec}(G)$ of a graph $G$ is the multiset of eigenvalues of its adjacency matrix. The fact that adjacency matrices are real and symmetric implies that their eigenvalues are real numbers. In particular, the spectrum of a graph is always real. We establish this fact in Subsection \ref{sub:Hermitian}.

\begin{example}\label{ex:Spectrum}

Consider once again the graph $G$ of Example \ref{ex:Path}. The adjacency matrix of $G$ is the $3\times 3$ symmetric matrix
\begin{equation*}
A(G)=\begin{bmatrix}
0 & 1 & 0\\
1 & 0 & 1\\
0 & 1 & 0
\end{bmatrix}.
\end{equation*} Moreover, the characteristic polynomial of $A(G)$ is given by
\begin{align*}
\varphi_{A(G)}(z)=\det\begin{bmatrix} z& -1 & 0\\
-1& z& -1\\
0 & -1 & z \end{bmatrix}=z^3-2z=z(z-\sqrt{2})(z+\sqrt{2}).
\end{align*} Therefore, the spectrum of $A(G)$ is given by $\operatorname{spec}(G)=\{\sqrt{2}, 0, -\sqrt{2}\}$.
\end{example}

\begin{example}
Consider the complete bipartite graph $K_{2,2}$ of Figure \ref{fig:Figure7}. The adjacency matrix of $K_{2,2}$ is
\begin{equation*}
A(K_{2,2})=\begin{bmatrix}
0 & 0 & 1 & 1 \\
0 & 0 & 1 & 1 \\
1 & 1 & 0 & 0 \\
1 & 1 & 0 & 0 
\end{bmatrix}.
\end{equation*} Moreover, the characteristic polynomial of $A(K_{2,2})$ is given by
\begin{align*}
\varphi_{A(K_{2,2})}(z)=\det\begin{bmatrix} z & 0 & -1 & -1 \\
0 & z & -1 & -1 \\
-1 & -1 & z & 0 \\
-1 & -1 & 0 & z \end{bmatrix}=z^4-4z^2=z^2(z+2)(z-2).
\end{align*} Therefore, the spectrum of $A(K_{2,2})$ is given by $\operatorname{spec}(K_{2,2})=\{2, 0, 0, -2\}$.
\end{example}


\subsection{The space of Hermitian matrices}\label{sub:Hermitian} Suppose $A\in \mathrm{M}_n$. The \emph{Hermitian conjugate} of $A$ is the matrix $A^{\ast}$ obtained by taking the complex conjugate of each entry in the transpose $A^{\tiny\mbox{T}}$ of $A$. An element $A\in \mathrm{M}_n$ is called \emph{Hermitian} if $A^{\ast}=A$ and the set of all $n\times n$ Hermitian matrices is denoted $\mathrm{H}_n$. We remark that $\mathrm{H}_n$ is a vector space over $\mathbb{R}$ of dimension $n^2$.

\begin{example}
The matrix $A=\begin{bmatrix} 1 & i\\ -i & 2\end{bmatrix}$ is Hermitian because $A^{\ast}=A$.
\end{example}

Suppose $G$ is a graph. Its adjacency matrix $A(G)$ is invariant under transposition. Moreover, the entries of $A(G)$ are invariant under complex conjugation since they are real numbers. Therefore, the adjacency matrix of a graph is always a Hermitian matrix. The following result can be found, for instance, in \cite[Theorem 12.6.1 (b)]{GarciaHorn} and establishes that the spectrum of any graph is indeed real. 

\begin{proposition}
The eigenvalues of a Hermitian matrix are real.
\end{proposition}

\begin{proof}
Suppose $A$ is Hermitian and $\mathbf{v}$ is an eigenvector for $A$ with eigenvalue $\lambda$. By scaling if necessary, we may assume $\langle\mathbf{v}, \mathbf{v}\rangle=1$. Observe 
\begin{align*}
\lambda = \langle \lambda \mathbf{v}, \mathbf{v} \rangle = \langle A\mathbf{v},\mathbf{v} \rangle = \langle \mathbf{v}, A^*\mathbf{v} \rangle = \langle \mathbf{v}, A\mathbf{v} \rangle = \langle \mathbf{v}, \lambda \mathbf{v} \rangle = \overline{\lambda}.
\end{align*} Therefore, $\lambda = \overline{\lambda}$, which implies $\lambda \in \mathbb{R}$. Here, $\langle\cdot, \cdot \rangle$ denotes the standard inner product on $\mathbb{C}^n$ defined by $\langle \mathbf{x} , \mathbf{y}\rangle=x_1\overline{y}_1+x_2\overline{y}_2+\cdots + x_n\overline{y}_n$ for vectors $\mathbf{x}=(x_1, x_2, \ldots, x_n)$ and $\mathbf{y}=(y_1, y_2, \ldots, y_n)$. 
\end{proof}

\subsection{CHS norms of important graphs} We consider the CHS norms for the families of paths, stars and complete graphs appearing in Subsection \ref{sub:Families}. The CHS norms of these families are important because they represent the extremal cases in Theorem \ref{thm:MainTheorem2}. We keep the convention of writing eigenvalues in nonincreasing order. Moreover, we write $\lambda_i^{[m]}$ when $\lambda_i$ has multiplicity $m$ with the convention $\lambda_i=\lambda_i^{[1]}$. The following results are found in \cite[Section 1.4]{Brouwer}.

\begin{proposition}\label{prop:Spectra} The spectra for the families $P_n$, $K_n$ and $K_{m,n}$ are as follows:\\

\noindent (a) $\operatorname{spec}(P_n)=\{2\cos\big( \frac{k\pi}{n+1}\big)\,:\,k=1,2,\ldots, n\}.$\\

\noindent (b) $\operatorname{spec}(K_n)=\{n-1, (-1)^{[n-1]}\}$.\\

\noindent (c) $\operatorname{spec}(K_{m,n})=\{\sqrt{mn}, 0^{[n+m-2]}, -\sqrt{mn}\}.$ 
\end{proposition}

Computing the CHS norms of $K_n$ and $K_{m,n}$ (and $S_n$) is a matter of appealing to the generating function \eqref{eq:MGF} for the CHS polynomials. The CHS norms of $K_{m,n}$ have a particularly nice expression. Evaluating the CHS polynomials on the spectrum of $P_n$ represents a difficult task, however. The CHS norms of $P_n$ therefore appear to have no simple expressions. We briefly revisit the CHS norms of paths in Section \ref{sec:PowerSums}.

\begin{proposition}\label{prop:Norms}
Suppose $d\geq 2$ is even. The CHS norms satisfy the following identities:\\

\noindent (a) $\norm{K_n}_d^d=\displaystyle \sum_{k=0}^d (-1)^k(n-1)^{d-k}{k+n-2\choose n-2} ,$\\

\noindent (b) $\norm{K_{m,n}}_d=\sqrt{mn}.$ 
\end{proposition}

\begin{proof}
Proposition \ref{prop:Spectra} (b) and the generating function of \eqref{eq:MGF} imply
\begin{align*}
\sum_{d=0}^{\infty} h_d(n-1, -1, -1, \ldots, -1)t^d=\frac{1}{1-(n-1)t}\cdot\Big( \frac{1}{1+t}\Big)^{n-1}.
\end{align*}Expand the factors in the right side into power series in $t$ to conclude
\begin{align*}
\frac{1}{1-(n-1)t}\cdot\Big( \frac{1}{t+1}\Big)^{n-1}=\sum_{d=0}^{\infty}(n-1)^dt^d \cdot\sum_{d=0}^{\infty}(-1)^d{d+n-2\choose n-2}t^d.
\end{align*} The Cauchy product formula for infinite power series ensures
\begin{align*}
\frac{1}{1-(n-1)t}\cdot\Big( \frac{1}{t+1}\Big)^{n-1}=\sum_{d=0}^{\infty}\Bigg( \sum_{k=0}^d (-1)^k(n-1)^{d-k}{k+n-2\choose n-2} \Bigg)t^d.
\end{align*} Compare coefficients to conclude (a). Proposition \ref{prop:Spectra} (c) and the generating function of \eqref{eq:MGF} imply
\begin{align*}
\sum_{d=0}^{\infty} h_d(\sqrt{mn}, 0, \ldots, 0, -\sqrt{mn})t^d=\frac{1}{1-\sqrt{mn}t}\cdot\frac{1}{1+\sqrt{mn}t} = \frac{1}{1-{mn}t^2}.
\end{align*} Expand the ride side of this relation into a geometric series to conclude
\begin{align}
\sum_{d=0}^{\infty} h_d(\sqrt{mn}, 0, \ldots, 0, -\sqrt{mn})t^d = \sum_{d=0}^{\infty} (mn)^dt^{2d}
\end{align} Compare the coefficients of each side to obtain the relation
\begin{align*}
h_d(\sqrt{mn}, 0, \ldots, 0 , -\sqrt{mn}) =
\begin{cases}
0 & d \text{ odd}, \\
(mn)^{\frac{d}{2}} & d \text{ even}.
\end{cases}
\end{align*} We conclude that $\norm{K_{m,n}}_d = \big((mn)^{\frac{d}{2}}\big)^{\frac{1}{d}} = \sqrt{mn}$ whenever $d\geq 2$ is even.
\end{proof}

\begin{corollary}
Suppose $d \geq 2$ is even. The CHS norm of $S_n$ satisfies $\norm{S_n}_d=\sqrt{n-1}$.
\end{corollary}

\begin{proof}
$S_n=K_{n-1,1}$. Proposition \ref{prop:Norms} (b) therefore ensures $\norm{S_n}_d=\sqrt{n-1}$.
\end{proof}

\section{Closed walks on graphs}\label{sec:ClosedWalks}

Suppose $G$ is a graph of order $n$ with vertex set $V=\{v_1, v_2, \ldots, v_n\}$ and edge set $E$. A \emph{walk} on $G$ is a sequence 
\begin{align*}
v_{i_1}e_{i_1i_2}v_{i_2}e_{i_2i_3}v_{i_3}e_{i_3i_4}\cdots e_{i_{k}i_{k+1}}v_{i_{k+1}}
\end{align*} of vertices and edges. The positive integer $k$ is called the \emph{length} of the walk. The vertices $v_{i_1}$ and $v_{i_{k+1}}$ are called the \emph{starting} and \emph{ending} vertices of the walk, respectively. A walk is \emph{closed} if the starting and ending vertices coincide.

\begin{example}
Consider the complete graph $K_4$. Figure \ref{fig:Figure9} shows the walk of length 2 defined by $v_1e_{13}v_3e_{32}v_2$. The arrows indicate the direction of the walk.
\end{example}

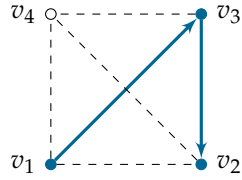
\begin{figure}[h]
\begin{center}
\begin{tikzpicture}[main_node/.style={circle,fill=white!20,draw,minimum size=0.4em,inner sep=1pt}, every loop/.style={}]

\node[main_node, label=left: $v_1$, color=MidnightBlue ] (1) at (0, 0) {};
\node[main_node, label=right: $v_2$, color=MidnightBlue ] (2) at (2,0) {};
\node[main_node, label=right: $v_3$, color=MidnightBlue ] (3) at (2, 2) {};
\node[main_node, label=left: $v_4$] (4) at (0, 2) {};

\path[every node/.style={}]
(1) edge[style=dashed] node {} (2)
(1) edge[color=MidnightBlue, style={->,> = latex'}, line width = 1.2pt] node {} (3) 
(1) edge[style=dashed] node {} (4) 
(2) edge[color=MidnightBlue, style={<-,> = latex'}, line width = 1.2pt] node {} (3)
(2) edge[style=dashed] node {} (4)
(3) edge[style=dashed] node {} (4) ;
\end{tikzpicture} 
\caption{A walk of length 2 on $K_4$.} 
\label{fig:Figure9}
\end{center}
\end{figure}

\begin{example}
Consider again the complete graph $K_4$. Figure \ref{fig:Figure10} shows two closed walks of length 3. Each walk touches the same vertices yet the walks are different. The starting vertices are colored pink to help distinguish between the walks.
\end{example}

\begin{figure}[h]
\begin{center}
\begin{tikzpicture}[main_node/.style={circle,fill=white!20,draw,minimum size=0.4em,inner sep=1pt}, every loop/.style={}]

\node[main_node, label=left: $v_1$, color=WildStrawberry ] (1) at (0, 0) {};
\node[main_node, label=right: $v_2$, color=MidnightBlue ] (2) at (2,0) {};
\node[main_node, label=right: $v_3$, color=MidnightBlue ] (3) at (2, 2) {};
\node[main_node, label=left: $v_4$] (4) at (0, 2) {};

\path[every node/.style={}]
(1) edge[color=MidnightBlue, style={<-,> = latex'}, line width = 1.2pt] node {} (2)
(1) edge[color=MidnightBlue, style={->,> = latex'}, line width = 1.2pt] node {} (3) 
(1) edge[style=dashed] node {} (4) 
(2) edge[color=MidnightBlue, style={<-,> = latex'}, line width = 1.2pt] node {} (3)
(2) edge[style=dashed] node {} (4)
(3) edge[style=dashed] node {} (4) ;
\end{tikzpicture}
\quad\quad\quad
\begin{tikzpicture}[main_node/.style={circle,fill=white!20,draw,minimum size=0.4em,inner sep=1pt}, every loop/.style={}]

\node[main_node, label=left: $v_1$, color=MidnightBlue ] (1) at (0, 0) {};
\node[main_node, label=right: $v_2$, color=MidnightBlue ] (2) at (2,0) {};
\node[main_node, label=right: $v_3$, color=WildStrawberry ] (3) at (2, 2) {};
\node[main_node, label=left: $v_4$] (4) at (0, 2) {};

\path[every node/.style={}]
(1) edge[color=MidnightBlue, style={<-,> = latex'}, line width = 1.2pt] node {} (2)
(1) edge[color=MidnightBlue, style={->,> = latex'}, line width = 1.2pt] node {} (3) 
(1) edge[style=dashed] node {} (4) 
(2) edge[color=MidnightBlue, style={<-,> = latex'}, line width = 1.2pt] node {} (3)
(2) edge[style=dashed] node {} (4)
(3) edge[style=dashed] node {} (4) ;
\end{tikzpicture}
\caption{Two distinct closed walks of length 3 on $K_4$.} 
\label{fig:Figure10}
\end{center}
\end{figure}
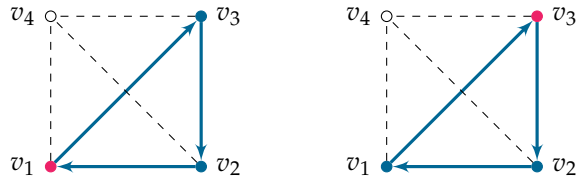

\begin{example}
A walk can have repeated edges. Figure \ref{fig:Figure11} shows the closed walk $v_1e_{13}v_3e_{31}v_1$ of length 2 on $K_4$. In particular, $e_{13}=e_{31}$ appears twice.
\end{example}

\begin{figure}[h]
\begin{center}
\begin{tikzpicture}[main_node/.style={circle,fill=white!20,draw,minimum size=0.4em,inner sep=1pt}, every loop/.style={}]

\node[main_node, label=left: $v_1$, color=WildStrawberry ] (1) at (0, 0) {};
\node[main_node, label=right: $v_2$ ] (2) at (2,0) {};
\node[main_node, label=right: $v_3$, color=MidnightBlue ] (3) at (2, 2) {};
\node[main_node, label=left: $v_4$] (4) at (0, 2) {};

\path[every node/.style={}]
(1) edge[style=dashed] node {} (2)
(1) edge[color=MidnightBlue, style={<->,> = latex'}, line width = 1.2pt] node {} (3) 
(1) edge[style=dashed] node {} (4) 
(2) edge[style=dashed] node {} (3)
(2) edge[style=dashed] node {} (4)
(3) edge[style=dashed] node {} (4) ;
\end{tikzpicture}
\caption{A closed walk of length 2 on $K_4$.}
\label{fig:Figure11}
\end{center}
\end{figure}
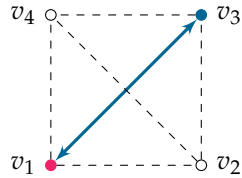

We are now in a position to define an important type of graph. A graph $G$ is called a \emph{tree} if there exists a unique walk (without repeated edges) starting at $v_i$ and ending at $v_j$ for any distinct vertices $v_i$ and $v_j$. The path and star graphs of Figures \ref{fig:Figure4} and \ref{fig:Figure5} are examples of trees. 




\subsection{Counting closed walks} How many closed walks of length 3 are there in $K_4$? Figure \ref{fig:Figure8} shows two closed walks of length 3. A brute force approach shows there are a total 12 closed walks of length 3. However, a brute force approach is no longer so feasible if we want to count the number of closed walks of length 7, for example. Fortunately, the number of walks $W_k(i,j)$ in a graph $G$ of length $k$ that start at $v_i$ and end at $v_j$ is elegantly encoded in the adjacency matrix of $G$. The following is \cite[Proposition 1.3.4]{Cvetkovic}.

\begin{proposition}\label{prop:Walks}
Suppose $G$ is a graph with adjacency matrix $A=A(G)$. The number of walks in $G$ of length $k$ that start at vertex $v_i$ and end at vertex $v_j$ satisfies the identity
\begin{align*}
W_k(i,j)=[A^k]_{ij}.
\end{align*}
\end{proposition}

\begin{proof}
We use induction on $k$ to prove that $W_k(i,j)=[A^k]_{ij}$. If $k=1$, then the claim holds by the definition of the adjacency matrix. If $k>1$, then our inductive hypothesis ensures
\begin{align*}
[A^k]_{ij}=\sum_{\ell=1}^n [A^{k-1}]_{i\ell}[A]_{\ell j}=\sum_{\ell=1}^n W_{k-1}(i,\ell) W_1(\ell,j)=W_k(i,j).
\end{align*}Indeed, $W_1(\ell, j)=1$ if and only if there is an edge between vertices $v_{\ell}$ and $v_{j}$.
\end{proof}

Suppose $G$ is a graph of order $n$ with spectrum $\operatorname{spec}(G)=\{\lambda_1, \lambda_2, \ldots, \lambda_n\}$ and adjacency matrix $A$. The number of closed walks of length $k$ in $G$ is denoted $c_k(G)$. Proposition \ref{prop:Walks} ensures
\begin{align*}
c_k(G)=\sum_{j=1}^n W_k(j,j)=\tr(A^k).
\end{align*} Any square matrix is similar to an upper triangular matrix \cite[Theorem 11.2.14]{GarciaHorn}. In particular, the trace of any matrix is given by the sum of its eigenvalues. The eigenvalues of $A^k$ correspond to the $k$th powers of the eigenvalues of $A$. Consequently, $\tr(A^k)=\lambda_1^k+\lambda_2^k+\cdots +\lambda_n^k$. The following is \cite[Theorem 3.1.1]{Cvetkovic}.

\begin{corollary}\label{cor:ClosedWalks}
Suppose $G$ is a graph with spectrum $\operatorname{spec}(G)=\{\lambda_1, \lambda_2, \ldots, \lambda_n\}$. The number of closed walks in $G$ of length $k$ satisfies the identity
\begin{align*}
c_k(G)=\lambda_1^k+\lambda_2^k+\cdots+\lambda_n^k.
\end{align*}
\end{corollary}

\begin{example}
Proposition \ref{prop:Spectra} (b) ensures $\operatorname{spec}(K_3) = \{2, -1,-1\}$. Therefore, 
\begin{align*}
c_4(K_3) = 2^4 + (-1)^4+(-1)^4= 18
\end{align*} In particular, there are a total of 18 closed walks of length 4 in $K_3$.
\end{example}

\begin{example}
Proposition \ref{prop:Spectra} (b) ensures $\operatorname{spec}(K_4)=\{3, -1, -1, -1\}$. Therefore,
\begin{align*}
c_7(K_4)=3^7+(-1)^7+(-1)^7+(-1)^7=2184.
\end{align*} In particular, there are a total of 2184 closed walks of length 7 in $K_4$.
\end{example}

\begin{remark}
In general, the number $c_k(G)$ is difficult to compute. In fact, estimates for $c_k(G)$ are often the best one can hope for. Bounds on the number of closed walks in a graph are given, for instance, in \cite{Chen, Nikiforov1}.
\end{remark}


\section{Power sum symmetric polynomials}\label{sec:PowerSums}

A polynomial $f(x_1,x_2, \ldots, x_n)$ with real coefficients is \emph{symmetric} if it is invariant under any permutation of $x_1, x_2, \ldots, x_n$. We say that $f(x_1, x_2, \ldots, x_n)$ is \emph{homogeneous of degree $d$} if the relation 
\begin{align*}
f(\alpha x_1, \alpha x_2, \ldots, \alpha x_n)=\alpha^k f(x_1, x_2, \ldots, x_n)
\end{align*} holds for all $\alpha\in \mathbb{R}$. The set of all degree-$d$ symmetric polynomials in $n$ variables is denoted $\Lambda_n^d$ and is a finite-dimensional vector space over $\mathbb{R}$. The reader can readily verify that $h_d(x_1, x_2, \ldots, x_n)$ is indeed an element of the vector space $\Lambda_n^d$.

The vector space $\Lambda_n^d$ has dimension equal to the number of partitions of $d$. A \emph{partition} of $d$ is a tuple $\boldsymbol{\pi}=(\pi_1, \pi_2, \ldots, \pi_{\ell})$ of positive integers such that $\pi_1\geq \pi_2\geq \cdots\geq \pi_{\ell}$ and $\pi_1+\pi_2+\cdots+\pi_{\ell}=d$ \cite[Section 1.7]{Stanley1}. We write $\boldsymbol{\pi}\vdash d$ whenever $\boldsymbol{\pi}$ is a partition of $d$. The number of partitions of $d$ is denoted $P(d)$.

\begin{example}
$P(3)=3$. In particular, the 3 partitions of $d=3$ are given by 
\begin{align*}
(3), \quad (2,1), \quad (1,1,1).
\end{align*}
\end{example}

\begin{example}
$P(4)=5$. In particular, the 5 partitions of $d=4$ are given by
\begin{align*}
(4), \quad (3,1), \quad (2,2), \quad (2,1,1), \quad (1,1,1,1).
\end{align*}
\end{example}

\begin{example}\label{ex:Six}
$P(6)=11$. There are 4 partitions of $d=6$ without 1 appearing:
\begin{align*}
(6), \quad (4,2), \quad (3,3), \quad (2,2,2).
\end{align*}
\end{example}

\begin{remark}
The first several terms of the sequence $\{P(d)\,: d\geq 1\}$ are given by
\begin{align*}
1, 2, 3, 5, 7, 11, 15, 22, 30, 42, 56, 77, 101, 135, 176, 231, 297, 385, \ldots ,
\end{align*}which is OEIS sequence A000041.\footnote{\url{https://oeis.org/A000041}} Interestingly, there is no closed form expression for $P(d)$. This is not to say we understand nothing about the sequence. For instance, Euler's pentagonal number theorem \cite[Theorem 14.3]{Apostol} can be used to prove the recursion $P(d)=\sum_k (-1)^{k-1}p(d-g_k)$. Here, the summation is taken over all nonzero integers and $g_k$ denotes the $k$th pentagonal number \cite[Page 5]{Apostol}.
\end{remark}


The \emph{power sum symmetric} polynomial of degree $d$ in the variables $x_1, x_2, \ldots, x_n$ is the homogeneous polynomial defined by 
\begin{align*}
p_d(x_1, x_2, \ldots, x_n)=x_1^d +x_2^d+\cdots +x_n^d.
\end{align*} We denote $p_d=p_d(x_1, x_2, \ldots, x_n)$ when $n$ is understood. The \emph{power sum symmetric function} for $\boldsymbol{\pi}=(\pi_1, \pi_2, \ldots, \pi_{\ell})\vdash d$ is the polynomial $p_{\boldsymbol{\pi}}\in \Lambda_n^d$ defined by
\begin{align*}
p_{\boldsymbol{\lambda}}=p_{\pi_1}p_{\pi_2}\cdots p_{\pi_{\ell}}.
\end{align*}

\begin{proposition}\label{prop:PowerSumBasis}
The set $\{p_{\boldsymbol{\pi}}\,:\,\boldsymbol{\pi}\vdash d\}$ is a basis for $\Lambda_n^d$. In particular, the dimension of $\Lambda_n^d$ is equal to the number of partitions of $d$.
\end{proposition}

Proposition \ref{prop:PowerSumBasis} is \cite[Corollary 7.7.2]{Stanley2}. We do not give a formal proof of this fact. However, we can briefly describe the idea of the proof. Essentially, one first establishes that a certain set $\{m_{\boldsymbol{\pi}}\,:\,\boldsymbol{\pi}\vdash d\}$ is a basis for $\Lambda_n^d$. The power sum symmetric functions are elements of $\Lambda_n^d$ and therefore admit and expansion of the form $p_{\boldsymbol{\pi}}=\sum_{\boldsymbol{\mu}\vdash d} R_{\boldsymbol{\pi} \boldsymbol{\mu}} m_{\boldsymbol{\mu}}$. The matrix $R $ is an invertible $P(d)\times P(d)$ matrix and therefore must map a basis to a basis. The mystery basis $\{m_{\boldsymbol{\pi}}\,:\,\boldsymbol{\pi}\vdash d\}$ is called the \emph{monomial basis} for $\Lambda_n^d$. See \cite[Section 7.3]{Stanley2}.

\begin{example}
The vector space $\Lambda_3^3$ has dimension $P(3)=3$ and basis given by $\{p_{(3)}, p_{(2,1)}, p_{(1,1,1)}\}$. In particular, any element of $\Lambda_3^3$ can be written as a unique linear combination of the polynomials 
\begin{align*}
x^3+y^3+z^3, \quad (x^2+y^2+z^2)(x+y+z), \quad (x+y+z)^3.
\end{align*}
\end{example}

\begin{example}
The vector space $\Lambda_3^4$ has dimension $P(4)=5$ and basis given by $\{p_{(4)}, p_{(3,1)}, p_{(2,2)}, p_{(2,1,1)}, p_{(1,1,1,1)}\}$. In particular, any element of $\Lambda_{3}^4$ can be written as a unique linear combination of the polynomials 
{\footnotesize
\begin{align*}
x^4+y^4+z^4, \, (x^3+y^3+z^3)(x+y+z), \, (x^2+y^2+z^2)^2, \,(x^2+y^2+z^2)(x+y+z)^2, \, (x+y+z)^4.
\end{align*}
}
\end{example}

\begin{remark}
The Schur polynomials are homogeneous symmetric polynomials indexed by partitions and with important connections to representation theory, algebraic combinatorics and algebraic geometry. Cauchy \cite{Cauchy} first defined the Schur polynomials as ratios of alternants. Jacobi \cite{Jacobi} later established a determinantal formula for the Schur polynomials which is now known as the Jacobi-Trudi identity. Finally, Schur \cite{Schur} established the intimate connection between Schur polynomials and the representation theory of the symmetric and general linear groups. Interestingly, the CHS polynomials are actually examples of Schur polynomials. See \cite[Sections 7.10 - 7.19]{Stanley2} for an introduction to Schur polynomials and their connection with combinatorics and representation theory.
\end{remark}

\subsection{The power sum expansion for CHS polynomials} Suppose $\boldsymbol{\pi}$ is a partition. Define the integer $z_{\boldsymbol{\pi}}=\prod_{i\geq 1}i^{m_i}m_i!$, in which $m_i$ denotes the multiplicity of $i$ in $\boldsymbol{\pi}$. The integer $z_{\boldsymbol{\pi}}$ is actually equal to the size of the centralizer of a permutation in the symmetric group of conjugacy class $\boldsymbol{\pi}$ \cite[Proposition 7.7.3]{Stanley2}.

\begin{example}\label{ex:Conjugacy2}
Consider the partition $(1,1)\vdash 2$. Observe $m_1=2$ since 1 appears twice in $\boldsymbol{\pi}$. Moreover, $m_i=1$ whenever $i>1$. Therefore, $z_{(1,1)}=1^{m_1}m_1!=1^22!=2.$ A similar calculation ensures $z_{(2)}= 2^{m_2}m_2!=2^1 1!=2$. Technically, the product defining $z_{\boldsymbol{\pi}}$ is infinite. However, the product contains only finitely many terms not equal to 1. 
\end{example}

\begin{example}\label{ex:Conjugacy46}
Values of $z_{\boldsymbol{\pi}}$ for the partitions of $4$ and $6$ (without 1 appearing) are shown in the following table:\\

\begin{center}
\begin{tabular}{ |c|c|} 
\hline
$\boldsymbol{\pi}$ & $z_{\boldsymbol{\pi}}$\\ 
\hline
(4) & 4\\
(3, 1) & 3\\
(2, 2) & 8\\
(2, 1, 1)& 4\\
(1, 1, 1, 1)& 24\\
\hline
(6) & 6\\
(4, 2)& 8 \\
(3, 3)& 18\\
(2, 2, 2) & 48\\
\hline
\end{tabular} 
\end{center} 
\end{example}

\vspace{1pc}
Again, we often denote $h_d=h_d(x_1, x_2, \ldots, x_n)$ when $n$ is understood. The CHS polynomial $h_d$ is an element of the vector space $\Lambda_n^d$. Therefore, $h_d$ can be expanded in the basis of power sum symmetric functions. The following result is \cite[Proposition 7.7.6]{Stanley2}. 

\begin{proposition}\label{prop:PowerSum}
The CHS polynomial $h_d$ admits the power sum expansion
\begin{align*}
h_d=\sum_{\boldsymbol{\pi}\vdash d} z_{\boldsymbol{\pi}}^{-1}p_{\boldsymbol{\pi}}.
\end{align*}
\end{proposition}

\begin{proof}
Our proof follows \cite[Page 25]{Macdonald}. Recall $\log(1-t)=-\sum_{k=1}^{\infty} \frac{t^k}{k}$ and $e^t=\sum_{k=0}^{\infty} \frac{t^k}{k!}$. Take the exponential of the logarithm of both sides of \eqref{eq:MGF} to conclude
\begin{align*}
\sum_{d=0}^{\infty}h_d(x_1, x_2, \ldots, x_n)t^d&=\exp\Big(-\sum_{i=1}^n\log(1-x_it)\Big)\\
&=\exp\Big(\sum_{i=1}^n\sum_{k=1}^{\infty} \frac{x_i^kt^k}{k}\Big)\\
&=\exp\Big(\sum_{k=1}^{\infty} \frac{(x_1^k+x_2^k+\cdots +x_n^k)t^k}{k}\Big)\\
&=\exp\Big(\sum_{k=1}^{\infty} \frac{p_k}{k} t^k\Big)\\
&=\prod_{k=1}^{\infty} \exp\Big( \frac{p_k}{k} t^k \Big) \\
&=\prod_{k=1}^{\infty} \sum_{m_k=0}^{\infty}\frac{(p_k t^k)^{m_k}}{k^{m_k}m_k!} \\
&=\sum_{d=0}^{\infty} \Big( \sum_{\boldsymbol{\pi}\vdash d} z_{\boldsymbol{\pi}}^{-1}p_{\boldsymbol{\pi}} \Big)t^d.
\end{align*} The result follows by comparing coefficients of $t^d$. 
\end{proof}

\begin{example}\label{ex:CHS2}
The partitions of $2$ are $(2)$ and $(1,1)$. Example \ref{ex:Conjugacy2} and Proposition \ref{prop:PowerSum} ensure that 
\begin{align*}
h_2=z_{(2)}^{-1}p_{(2)}+z_{(1,1)}^{-1}p_{(1,1)}=\frac{p_2}{2}+\frac{p_1^2}{2}.
\end{align*}

\end{example}

\begin{example}\label{ex:CHS4}
The partitions of $4$ are $(4)$, $(3,1)$, $(2,2)$, $(2,1,1)$ and $(1,1,1,1)$. Example \ref{ex:Conjugacy46} and Proposition \ref{prop:PowerSum} ensure that 
\begin{align*}
h_4=\sum_{\boldsymbol{\pi}\vdash 4} z_{\boldsymbol{\pi}}^{-1}p_{\boldsymbol{\pi}}=\frac{p_4}{4}+\frac{p_3p_1}{3}+\frac{p_2^2}{8}+\frac{p_2p_1^2}{4}+\frac{p_1^4}{24}.
\end{align*}
\end{example}

\begin{proposition}\label{prop:246Norms}
Suppose $G$ is a graph with $m$ edges and adjacency matrix $A$. The following identities hold:\\

\noindent (a) $\norm{G}_2^2=\frac{1}{2} \tr(A^2)=m$,\\

\noindent (b) $\norm{G}_4^4=\frac{1}{4}\tr(A^4)+\frac{1}{8}\big[ \tr (A^2)\big]^2=\frac{1}{4}\tr(A^4)+\frac{1}{2}m^2,$\\

\noindent (c) $\norm{G}_6^6=\frac{1}{6}\tr(A^6)+\frac{1}{4}m\tr(A^4)+\frac{1}{18}\big[\tr(A^3)\big]^2+\frac{1}{6}m^3.$
\end{proposition}

\begin{proof}
Suppose $G$ has eigenvalues $\lambda_1, \lambda_2, \ldots, \lambda_n$. The trace of any matrix is the sum of its eigenvalues. The adjacency matrix of any graph has an all zero diagonal. Therefore, $p_1(\lambda_1, \lambda_2, \ldots, \lambda_n)=0$. Example \ref{ex:CHS2} implies
\begin{align*}
\norm{G}_2^2=h_2(\lambda_1, \lambda_2, \ldots, \lambda_n)=\frac{1}{2}\Big(\lambda_1^2+\lambda_2^2+\cdots +\lambda_n^2\Big).
\end{align*} However, Corollary \ref{cor:ClosedWalks} ensures $\lambda_1^2+\lambda_2^2+\cdots +\lambda_n^2=c_2(G)=2m$, which establishes part (a). However, we remark that $p_k(\lambda_1, \lambda_2, \ldots, \lambda_n)=\tr(A^k)$. Part (b) follows by the same argument applied to the expression for $h_4$ of Example \ref{ex:CHS4}. Part (c) follows by expanding $h_6$ in the power sum basis using the coefficients $z_{\boldsymbol{\pi}}$ given in Example \ref{ex:Conjugacy46}. Again, we remark that only the partitions without a 1 appearing are needed in the power sum expansion for $h_6$ because $p_1(\lambda_1, \lambda_2, \ldots, \lambda_n)=0$.
\end{proof}

The path of order $n$ has $m=n-1$ edges. Proposition \ref{prop:246Norms} (a) therefore implies that $\norm{P_n}_2=\sqrt{n-1}$. However, computing $\norm{P_n}_d$ for $d\geq 4$ seems unattainable at the moment. We leave this is an interesting open problem.\\

\noindent\textbf{Problem 1}. Compute $\norm{P_n}_d$ for $d\geq 4$.

\section{Proof of Theorem \ref{thm:MainTheorem1}}\label{sec:Proof1}

Suppose $G$ and $H$ are noncospectral singularly cospectral graphs of order $n$. Let $\operatorname{spec}(G)=\{\lambda_1, \lambda_2, \ldots, \lambda_n\}$ and $\operatorname{spec}(H)=\{\mu_1, \mu_2, \ldots, \mu_n\}$ denote the spectra of $G$ and $H$. The fact that $G$ and $H$ are simple graphs implies that
\begin{align}
p_1(\boldsymbol{\lambda})=p_1(\boldsymbol{\mu})=0.\label{eq:Proof1A}
\end{align} Moreover, the fact that $G$ and $H$ have the same singular values implies that
\begin{align}
p_k(\boldsymbol{\lambda})=p_k(\boldsymbol{\mu})\label{eq:Proof1B}
\end{align} when $k$ is even. Graphs $G$ and $H$ are cospectral if and only if $p_k(\boldsymbol{\lambda})=p_k(\boldsymbol{\mu})$ for all $k$. This fact follows from \cite[Lemma 2.1]{Godsil}. Consequently, there exists a smallest odd integer $j\neq 1$ for which $p_j(\boldsymbol{\lambda})\neq p_j(\boldsymbol{\mu})$. Now consider the set $\{\boldsymbol{\pi}\,:\, \boldsymbol{\pi}\vdash d\}$ of partitions of $d=j+3$. The observations in \eqref{eq:Proof1A} and \eqref{eq:Proof1B} imply $p_{\boldsymbol{\pi}}(\boldsymbol{\lambda})=p_{\boldsymbol{\pi}}(\boldsymbol{\mu})$ if and only if $\boldsymbol{\pi}\neq (j, 3)$. Proposition \ref{prop:PowerSum} now ensures that $\norm{G}_d^d\neq \norm{H}_d^d$ if and only if $p_{(j,3)}(\boldsymbol{\lambda})\neq p_{(j,3)}(\boldsymbol{\mu})$. However, $p_{(j,3)}=p_jp_3$ and $j$ was chosen to satisfy $p_j(\boldsymbol{\lambda})\neq p_j(\boldsymbol{\mu})$. Theorem \ref{thm:MainTheorem1} follows.


\section{Proof of Theorem \ref{thm:MainTheorem2}}\label{sec:Proof2}

Suppose $G$ is a graph. Recall that $c_k(G)$ denotes the number of closed walks in $G$ of length $k$. We denote $c_k=c_k(G)$ when $G$ is understood. The following result immediately follows from Corollary \ref{cor:ClosedWalks} and Proposition \ref{prop:PowerSum}.

\begin{proposition}\label{prop:NormClosed}
Suppose $d\geq 2$ is even. The CHS $d$-norm of a graph $G$ satisfies
\begin{align}
\norm{G}_d^d=\sum_{\boldsymbol{\pi}\vdash d} z_{\boldsymbol{\pi}}^{-1}c_{\boldsymbol{\pi}}, \label{eq:Key}
\end{align}in which $c_{\boldsymbol{\pi}}=c_{\pi_1}c_{\pi_2}\cdots c_{\pi_{\ell}}$ for a partition $\boldsymbol{\pi}=(\pi_1, \pi_2, \ldots, \pi_{\ell})$.
\end{proposition}

The coefficients $z_{\boldsymbol{\pi}}^{-1}$ appearing in \eqref{eq:Key} are positive. Proposition \ref{prop:NormClosed} therefore turns the problem of finding extremal graphs for CHS norms into finding extremal graphs for $c_k$. In particular, what graphs minimize and maximize $c_k$? In general, this is not an easy question. Csikv\'ari, however, solved this problem for simple connected graphs and trees \cite{Csikvari}. In particular, \cite[Theorem 4.6]{Csikvari} establishes that $c_k(G)$ is minimized (for each $k$) over all simple connected graphs of order $n$ by the path $P_n$. Consequently, the lower bounds in Theorem \ref{thm:MainTheorem2} (a) and (b) hold. The upper bound in Theorem \ref{thm:MainTheorem2} (a) is immediate from the following proposition.

\begin{proposition}
If $G$ is a simple connected graph of order $n$, then $c_k(G)$ is maximized by the complete graph $K_n$ for each $k$.
\end{proposition}

\begin{proof}
The graph $G$ can be embedded inside the complete graph $K_n$ as a subgraph. Therefore, any closed walk in $G$ can be viewed as a closed walk in $K_n$. Consequently, $c_k(G)\leq c_k(K_n)$ for each $k$. 
\end{proof}

\cite[Theorem 4.6]{Csikvari} also establishes that $c_k(T)$ is maximized (for each $k$) over all trees of order $n$ by the star $S_n$. Consequently, the upper bound in Theorem \ref{thm:MainTheorem2} (b) holds. And this completes the proof of Theorem \ref{thm:MainTheorem2}.


\section{Open questions and closing remarks}\label{sec:Closing}

Norms on a finite-dimensional vector space are equivalent. For instance, the CHS norms satisfy the bound $\norm{A}_d^d\leq {n+d-1\choose d}\norm{A}^d$ for all $A\in \mathrm{H}_n$ \cite[Theorem 38]{Aguilar}. Here, $\norm{A}$ denotes the spectral norm of $A$. Moreover, equality holds if and only if $A=I_n$. However, the above bound is no longer sharp if we consider the CHS norms of graphs precisely because $I_n$ is not the adjacency matrix of a simple graph. The following natural problem appears difficult to answer.\\

\noindent\textbf{Problem 2.} Find sharp bounds for $\norm{G}_d$ in terms of $\norm{G}_{\ast}$, $\norm{G}$ and $\norm{G}_{S_p}$.\\


Ultimately, our paper is about finding the extrema of $\norm{\cdot}_d$ over a given family of graphs. In particular, we have found the extrema of the CHS norms over the families of simple connected graphs and trees. But what about other families?\\

\noindent\textbf{Problem 3.} Can one find the extrema of the CHS norms over other families of graphs such as regular, unicyclic and complete multipartite graphs?


\end{document}